\documentclass[10pt]{article}
\title{A Proximal Gradient Method for Nonconvex Nonsmooth Optimization with Applications} 
\author{Jos\'e M.M. Brito\thanks{School of Science, Great Bay University,
Dongguan 523000, Guangdong Province, People's Republic of China.
E-mail: marciomb@ufpi.edu.br. ORCID-ID: 0000-0003-4362-0536}
\and
Felipe Lara\thanks{Instituto de Alta Investigaci\'on (IAI), Universidad de
Tarapac\'a, Arica, Chile. E-mail: felipelaraobreque@gmail.com;
flarao@academicos.uta.cl.
Web: felipelara.cl, ORCID-ID: 0000-0002-9965-0921}
\and
Jinyun Yuan\thanks{School of Science, Great Bay University,
Dongguan 523000, Guangdong Province, People's Republic of China;
School of Computer Science and Technology, Dongguan University of Technology,
Dongguan 523808, Guangdong Province, People's Republic of China.
Corresponding author. E-mail: yuanjy570522@126.com.
ORCID-ID: 0000-0003-0358-2862}
}

\usepackage{amsmath}
\usepackage{amsthm}
\usepackage{amssymb}
\usepackage{multicol}
\usepackage[active]{srcltx}
\usepackage{algorithm}
\usepackage{array}
\usepackage[dvipsnames,svgnames]{xcolor}
\usepackage{url}
\usepackage{tikz}
\usepackage{graphicx}
\usepackage{subcaption} 
\usepackage{pgfplots}
\usepackage{tabularx}
\usepackage{booktabs} % for better tables
\usepackage[normalem]{ulem}
\usepackage{comment}
\usepackage[hidelinks]{hyperref}
\providecommand{\U}[1]{\protect \rule{.1in}{.1in}}
\newtheorem{theorem}{Theorem}[section]

\newtheorem{corollary}[theorem]{Corollary}

\theoremstyle{definition}
\newtheorem{definition}[theorem]{Definition}
\theoremstyle{plain}
\newtheorem{example}[theorem]{Example}

\newtheorem{lemma}[theorem]{Lemma}

\newtheorem{proposition}[theorem]{Proposition}
\theoremstyle{remark}
\newtheorem{remark}[theorem]{Remark}
\theoremstyle{plain}

\pgfplotsset{compat=1.18} 

\newcommand\R{\mathbb{R}}
\newcommand{\argmin}{{\rm argmin}}

\begin{document}

%\today

\maketitle

\begin{abstract}
\noindent We study a proximal-gradient method for the sum of a convex smooth function $g$ and a possibly nonconvex, nonsmooth, strongly quasar-convex regularizer. The gradient step is weighted by the quasar-convexity parameter, so, for step sizes in the descent range, the iteration decreases a weighted objective, while the unweighted sum serves for comparison. We establish a geometric envelope around every global minimizer of the unweighted objective, with an explicit radius independent of the step size; a common minimizer yields global linear convergence. We also prove that every critical point $u$ of the weighted objective lies within an explicit multiple of $\|\nabla g(u)\|$ from every global minimizer of the unweighted objective. Together with an optimized composite comparison, this yields asymptotic distance bounds. An example shows that the two objectives can have different critical points. For power-regularized least squares with powers between zero and one, descent provides a confinement ball for the sharp local modulus, and support identification combined with a KL argument yields whole-sequence convergence.

\medskip

\noindent{\small \emph{Keywords}: Nonconvex optimization; Proximal gradient algorithms; Quasar-convexity; Nonconvex regularization; Least squares}

\medskip

\noindent {\bf Mathematics Subject Classification:} 90C26; 90C30.
\end{abstract}

\section{Introduction}

Composite optimization provides a natural formulation for problems in which a smooth data-fitting term is combined with a nonsmooth regularizer. Proximal-gradient methods exploit this structure by separating a gradient step from a proximity step. Nonconvex penalties are of particular interest in sparse approximation, where they induce different sparsity patterns and estimation biases from convex regularizers.  Their analysis requires care: the proximity operator can be
set-valued, stationary points need not be globally optimal, and guarantees derived from convexity are no longer available~\cite{ABS,BB,BST,LY}.

Quasar-convexity offers a different source of structure by controlling a function along segments directed toward a global minimizer. Smooth quasar-convex optimization admits efficient first-order methods~\cite{GGK,HSS,WW}. For nonsmooth functions, the proximal-point analysis in \cite{BLL} establishes global convergence and linear rates under strong quasar-convexity. The proximal comparison inequality relative to the quasar-convex center in \cite[Proposition~15]{BLL} is a starting point for the present work. Proximal quasar-convexity has also been introduced for composite models with convex regularizers and random zeroth-order oracles~\cite{FPBS}. Our setting differs because the regularizer itself may be nonconvex and its quasar-convex center need not minimize the full composite objective.

The composite analysis in~\cite[Section~3.3]{HADR} considers an
accelerated proximal-gradient method for a strongly quasar-convex
smooth term and a convex nonsmooth term. The composite convergence
result is restricted to the quasar-convexity parameter equal to
one ($\kappa=1$ in our notation) and requires the smooth term to
satisfy the extended strong quasar-convexity inequality at a global
minimizer of the sum, together with a lower-curvature condition.
The location of this reference point is essential, since adding
a convex term may displace the minimizer from the original
quasar-convex center. In our setting, the smooth term is convex,
while the regularizer may be nonconvex and nonsmooth. We retain
its own quasar-convex center and derive proximal comparisons
valid at arbitrary reference points.

We analyze the splitting iteration
\[
 x_{k+1}\in {\rm Prox}_{\beta_k h}
 \bigl(x_k-\kappa\beta_k\nabla g(x_k)\bigr),
 \qquad 0<\kappa\leq1,
\]
where $g$ is convex with a Lipschitz continuous gradient and $h$ is proper, lower semicontinuous, and strongly quasar-convex. The factor $\kappa$ in the gradient step is chosen to match the
coefficient arising from the proximal comparison for a strongly
quasar-convex regularizer. Both component comparisons then carry
the factor $2\kappa\beta_k$, and their sum yields estimates relative
to global minimizers of $f=g+h$, which need not coincide with the
quasar-convex center of $h$. This choice allows us to exploit the geometry of the regularizer without imposing quasar-convexity on the composite objective. For steps in the descent range, the iteration decreases $F_\kappa=\kappa g+h$, so its critical-point convergence concerns this weighted objective. We thus distinguish
the objective decreased by the algorithm from the unweighted objective used in the global-minimizer comparisons.

Our first main result combines geometric and asymptotic distance estimates relative to global minimizers of $f$. For every constant step in the descent range, a geometric envelope controls the distance to a ball centered at any such minimizer, with an explicit radius independent of the step size.
A common minimizer of $g$ and $h$ yields global linear convergence for every proximal selection. Our second main result is a comparison at critical points: every critical point $u$ of $F_\kappa$ lies within an explicit multiple of $\|\nabla g(u)\|$ from every global minimizer of $f$. Optimizing the free parameter in the composite inequality recovers this same comparison constant when $\mu\leq\kappa \gamma/(2-\kappa)$, where $\mu$ and $\gamma$ denote, respectively, the strong-convexity modulus of $g$ and the quasar-convexity modulus of $h$ (Assumption~(H), Section~\ref{sec:03}); for larger values of $\mu$, the critical-point comparison supplies the stronger asymptotic estimate. An example with a nonconvex regularizer shows that critical points of $F_\kappa$ need not be critical for $f$.

For the application $h(x)=\rho\|x\|_p^p$, $0<p<1$, strong quasar-convexity is available only on bounded sets. We recall its sharp modulus on a Euclidean ball and derive the required ball directly from the initial weighted objective value, resolving the confinement requirement without modifying the algorithm. To complete the application, we apply established thresholding and restricted-KL arguments \cite{ABS,ZLWX,ZLX} to verify whole-sequence convergence to a critical point of $F_\kappa$. These properties complement the comparison estimates and do not imply global sparse recovery.

Section~\ref{sec:02} presents the preliminaries, and Section~\ref{sec:03} develops the proximal comparison inequalities and convergence estimates. Section~\ref{sec:04} applies the results to power-regularized least squares, followed by the conclusions in Section~\ref{sec:conclusions}.

\section{Preliminaries and Basic Definitions} \label{sec:02}

The inner product on $\mathbb{R}^{n}$ and the Euclidean norm are denoted by $\langle \cdot,\cdot \rangle$ and $\| \cdot \|$, respectively. The open ball with center at $x_{0} \in \mathbb{R}^{n}$ and radius $\varepsilon > 0$ is denoted by $\mathbb{B}(x_{0}, \varepsilon)$. {For $r\geq0$, we write $\overline{\mathbb B}(x_0,r):=\{x\in\mathbb R^n:\|x-x_0\|\leq r\}$ for the closed ball, and also use $B$ in place of $\mathbb B$. We set $(t)_+:=\max\{t,0\}$ for $t\in\mathbb R$.} We write $\mathbb{R}_{+}:=[0, + \infty[$ and $\mathbb{R}_{++} := \, ]0, + \infty[$. For a nonempty set $K\subset\mathbb{R}^{n}$, we denote its closure by $\overline{K}$. 
%its boundary by ${\rm bd}\,K$, its topological interior by ${\rm int}\,K$ 
%its convex hull by ${\rm conv}\,K$, 
%Given a convex and closed set $K$, the projection of $x$ on $K$ is denoted by $P_{K} (x)$, and the indicator function on $K$ by $\iota_{K}$. 
%The closed ball with cen\-ter at $x$ and radius $\delta > 0$ is denoted by $\mathbb{B} (x, \delta)$. 
%and the vector with $1$ in the $i$th coordinate and $0$'s elsewhere is denoted by $e_{i}$.
%and its asymptotic (recession) cone by $K^{\infty} := \{ u \in \mathbb{R}^{n}: ~ \exists ~ t_{k} \rightarrow + \infty, ~ \exists~ x_{k} \in K, ~ \frac{x_{k}}{t_{k}} \rightarrow u \}$. %Furthermore, for any set $K \subset \mathbb{R}^{n}$, it follows from \cite[Proposition 2.1.2]{AT} that
%\begin{equation}\label{basic:rec}
%$K^{\infty}=\{0\}$ if and only if $K$ is bounded.
%\end{equation}

A nonempty set $K \subset \mathbb{R}^{n}$ is said to be star-shaped at $x_{0} \in K$ if, for every $x \in K$, the segment $[x_{0}, x] \subset K$.

Given any $u,v, w\in \mathbb{R}^{n}$ and any $\theta \in \mathbb{R}$, the following relations hold: 
\begin{align}
 & \langle u - w,\, v - u \rangle
   = \frac{1}{2}\| w - v\|^{2}
   - \frac{1}{2}\| u - w\|^{2}
   - \frac{1}{2}\| v - u\|^{2},
   \label{3:points} \\[4pt]
 & \| \theta u + (1-\theta)v\|^{2}
   = \theta\| u\|^{2}
   + (1 - \theta)\| v\|^{2}
   - \theta(1 - \theta)\| u - v\|^{2}.
   \label{iden:1}
\end{align}

Given any extended-valued function $h: \mathbb{R}^{n} \rightarrow \overline{\mathbb{R}}$, the effective domain of $h$ is defined by ${\rm dom}\,h := \{x \in \mathbb{R}^{n}: h(x) < + \infty \}$. The function $h$ is proper if ${\rm dom}\,h$ is nonempty and $h(x) > - \infty$ for all $x \in \mathbb{R}^{n}$. The notion of properness is important when dealing with minimization pro\-blems.

We denote by
$S_{\lambda}(h):=\{x\in\mathbb{R}^{n}:h(x)\leq\lambda\}$
the sublevel set of $h$ at level $\lambda\in\mathbb{R}$,
and by $\argmin_{\mathbb{R}^{n}}h$ the set of its global minimizers. A function $h$ is lower semicontinuous (lsc) at $\overline{x} \in \mathbb{R}^{n}$ if, for every sequence $\{x_k\}_{k} \subset \mathbb{R}^{n}$ with $x_k \rightarrow \overline{x}$, we have $h(\overline{x}) \leq \liminf_{k \rightarrow + \infty} h(x_k)$. %Furthermore, the usual con\-ven\-tion  $\sup_{\emptyset} h := - \infty$ and $\inf_{\emptyset} h := + \infty$ is adopted.

A proper function $h: \mathbb{R}^{n} \rightarrow \overline{\mathbb{R}}$ is coercive if
\[
 \lim_{\|x\|\rightarrow+\infty} h(x)=+\infty,
\]
or equivalently, if $S_\lambda(h)$ is bounded for every $\lambda \in \mathbb{R}$.

The following generalized convexity notion has been used in several applications in machine learning {because it supports convergence guarantees for gradient methods} (see \cite{GGK,HMR,HADR,HSS,WW} and references therein).

Motivated by \cite[Lemma 10]{HSS}, we define quasar-convexity as follows:

\begin{definition}\label{def:quasar}
 Let $\kappa \in \, ]0, 1]$, $\gamma \geq 0$, let $h: \mathbb{R}^n \to \overline{\mathbb{R}}$ be proper, and let $K\subset\mathbb{R}^n$ be star-shaped at $\overline{x}\in K\cap \argmin_{ \mathbb{R}^n}\,h$. We say that $h$ is $(\kappa, \gamma)$-strongly quasar-convex on $K$ with respect to $\overline{x}$ if
\[
 h(\lambda\overline{x}+(1-\lambda)x) \leq \kappa \lambda h(\overline{x}) + (1-\kappa\lambda) h(x) - \lambda \left(1-\frac{\lambda}{2-\kappa} \right) \frac{\kappa\gamma}{2}\|x-\overline{x}\|^2
\]
for all $\lambda\in[0,1]$ and all $x\in K\cap{\rm dom}\,h$. When $K=\mathbb{R}^n$, we omit the reference to $K$; when $\gamma=0$, we say that
$h$ is $\kappa$-quasar-convex.
\end{definition}

In particular, $K\cap{\rm dom}\,h$ is star-shaped at $\overline{x}$, since the
right-hand side is finite for $x\in K\cap{\rm dom}\,h$.

If $h$ is $(\kappa,\gamma)$-strongly quasar-convex on $K$ with respect to $\overline{x}$ and $\gamma > 0$, then $\overline{x}$ is the unique minimizer of $h$ on $K$, by the quadratic-growth inequality \eqref{qwc:sconvex} below; in the global case, $\argmin_{\mathbb{R}^n}h=\{\overline{x}\}$. If $\gamma=0$, then $K\cap\argmin_{\mathbb{R}^n}h$ is star-shaped at $\overline{x}$ (see \cite[Observation 3]{HSS}). For $\kappa=1$, (strongly) quasar-convex functions are (strongly) star-convex \cite{NP-2006}. In particular, every (strongly) convex function with modulus $\gamma\geq0$ that attains its minimum is $(1,\gamma)$-strongly quasar-convex with respect to any of its minimizers.

The relationship between convex/star-convex/quasar-convex notions is su\-mma\-ri\-zed below:
 \begin{align}\label{scheme}
  \begin{array}{ccccccc}
  {\rm strongly ~ convex} & \overset{*}{\Longrightarrow} & {\rm strongly ~ star{\rm -}convex} & \overset{*}{\Longrightarrow} & {\rm strongly ~ quasar{\rm -}convex}  \notag \\
  \Downarrow & \, & \Downarrow & \, & \Downarrow  \notag \\
  {\rm convex} & \overset{*}{\Longrightarrow} & {\rm star{\rm -}convex} & \overset{*}{\Longrightarrow} & {\rm quasar{\rm -}convex},
  \end{array}
 \end{align}
 where ``$*$" denotes that $\argmin_{\mathbb{R}^{n}\,}h \neq \emptyset$ is required. {None of the reverse implications holds in general} (see \cite{GGK,HADR,HSS,NP-2006}).

\begin{remark}%\label{not:included}
 {Neither (strong) quasiconvexity nor (strong) quasar-convexity implies the other in general}. Indeed, the function in {\cite[Example 11]{HLMV}} is strong\-ly quasiconvex but not quasar-convex, while the function in \cite[Example 8]{BLL} is strongly quasar-convex {on a ball in dimension at least two} but not quasiconvex. For differentiable func\-tions, sufficient conditions under which strong quasiconvexity implies quasar-convexity are provided in \cite{HLMV,LV}.
\end{remark}

A function $h: \mathbb{R}^{n} \rightarrow \mathbb{R}$ is said to be $L$-smooth on an open set $K \subset \mathbb{R}^{n}$ if it is differentiable on
$K$ and
\begin{equation}\label{L:smooth}
 \| \nabla h(x) - \nabla h(y) \| \leq L \| x - y \|, ~ \forall ~ x, y \in K.
\end{equation}

For $L$-smooth functions, a fundamental result is the descent lemma: if $h$ is $L$-smooth on a convex set $K$, with $L\geq0$, then for every $x,y\in K$, we have
\begin{equation}\label{descent:lemma}
 h(y) \leq h(x) + \langle \nabla h(x), y - x \rangle + \frac{L}{2} \| x - y \|^{2}.
\end{equation}

Also, if $h$ is $(\kappa,\gamma)$-strongly quasar-convex with $\gamma>0$, then {using $h(\overline{x})\leq h(\lambda\overline{x}+(1-\lambda)y)$ in Definition~\ref{def:quasar} and letting $\lambda\uparrow1$ gives}
\begin{equation}\label{qwc:sconvex}
 h(\overline{x}) + \frac{\kappa \gamma}{2(2 - \kappa)} \| y - \overline{x} \|^{2} \leq h(y), ~ \forall ~ y \in K,
\end{equation}
where $\overline{x} \in \argmin_{\mathbb{R}^{n}}\,h$. Thus, strongly quasar-convex functions satisfy a quadratic growth condition with modulus $\frac{\kappa \gamma}{2(2 - \kappa)} > 0$. %a modulus which reduces to the strongly convex one $\frac{\gamma}{2}$ when $\kappa = 1$.

In the differentiable case, we have the following useful characterization: a differentiable function $h: \mathbb{R}^{n} \rightarrow \mathbb{R}$ is $(\kappa, \gamma)$-strongly quasar-convex on $\mathbb{R}^n$ ($\kappa \in \, ]0, 1]$ and $\gamma \geq 0$) with respect to $\overline{x} \in \argmin_{\mathbb{R}^{n}}\,h$ if and only if (see \cite[Lemma 10]{HSS})
\begin{equation}\label{diff:squasar}
 h(\overline{x}) \geq h(y) + \frac{1}{\kappa} \langle \nabla h(y), \overline{x} - y \rangle + \frac{\gamma}{2} \| y - \overline{x} \|^{2}, ~ \forall ~ y \in \mathbb{R}^{n}.
\end{equation}
 Hence, quasi-strongly convex functions \cite{NNG} are related to strongly quasar-convex functions.
 Further properties of differentiable (strongly) quasar-convex functions may be found in \cite{GGK,HMR,HADR,HSS,WW}, among others.

Let $K\subset\mathbb R^n$ be nonempty, closed and convex, let $h:\mathbb R^n\to\overline{\mathbb R}$ be proper and lsc, and let $\beta>0$. The proximity operator ${\rm Prox}_{\beta h}(K,\cdot):\mathbb R^n\rightrightarrows\mathbb R^n$ of $h$ on $K$ with parameter $\beta$ is defined by
\begin{equation}\label{prox:operator}
 {\rm Prox}_{\beta h} (K, z) = \mathrm{argmin}_{x \in K} \left\{ h(x) + \frac{1}{2 \beta} \| z - x \|^{2} \right\}.
\end{equation}
When $K = \mathbb{R}^{n}$, we simply write ${\rm Prox}_{\beta h} (z) := {\rm Prox}_{\beta h} (\mathbb{R}^{n}, z)$. If $h$ is proper, lsc and convex, then ${\rm Prox}_{\beta h}$ is single-valued (see, for instance, \cite[Proposition
12.15]{BC-2}).

{
For $x\in\operatorname{dom}h$, the Fr\'echet subdifferential is
\[
 \widehat\partial h(x)=\left\{v\in\mathbb R^n:
 \liminf_{\substack{z\to x\\z\ne x}}
 \frac{h(z)-h(x)-\langle v,z-x\rangle}{\|z-x\|}\geq0\right\}.
\]
The limiting subdifferential consists of limits $v_j\to v$ with
$v_j\in\widehat\partial h(x_j)$, $x_j\to x$ and $h(x_j)\to h(x)$.
}

The following notion extends the classical concept of a stationary point to the possibly nonsmooth setting considered here. 

Given a proper lsc function $h: \mathbb{R}^{n} \rightarrow \overline{\mathbb{R}}$, we denote by $\partial^{L} h(x)$ its (limiting) subdifferential at $x \in {\rm dom}\,h$ (see {\cite{RW} for the limiting subdifferential}, and \cite{ABS} and the references therein for the general lsc case, which covers functions such as $\lvert x \rvert^{p}$, $0<p<1$, failing to be locally Lipschitz at the origin). 

Recall that $\partial^{L} h(x) = \{\nabla h(x)\}$ whenever $h$ is {continuously differentiable near $x$}, that the sum rule
\[
 \partial^{L} (\phi + h)(x) = \nabla \phi(x) + \partial^{L} h(x)
\]
holds for every $\phi \in C^{1}(\mathbb{R}^{n})$ and every $x \in {\rm dom}\,h$, and that $\partial^{L} h$ has the following closedness property: if $x_{k} \rightarrow x$, {$h(x_k)\rightarrow h(x)$}, $w_{k} \in \partial^{L} h(x_{k})$ and $w_{k} \rightarrow w$, then $w \in \partial^{L} h(x)$.

Given $\kappa \in \, ]0, 1]$, a point $\overline{x} \in \mathbb{R}^{n}$ is said to be a $\kappa$-critical point of $f = g + h$ if
\begin{equation}\label{crit:point}
 0 \in \kappa \nabla g(\overline{x}) + \partial^{L} h(\overline{x}),
\end{equation}
that is, $\overline{x}$ is a critical point, in the sense above, of $F_{\kappa} := \kappa g + h$. When $\kappa = 1$, \eqref{crit:point} reduces to the usual notion of a critical point of $f = g + h$.

\begin{definition}[Kurdyka--\L{}ojasiewicz property]\label{def:KL}
Let $\sigma:\mathbb R^n\to\overline{\mathbb R}$ be proper and lsc, and write
\[
 \operatorname{dist}(0,\partial^L\sigma(u)):=\inf\{\|v\|:v\in\partial^L\sigma(u)\}
\]
(with $\inf\emptyset:=+\infty$). For $\eta\in\,]0,+\infty]$, let $\Phi_\eta$ denote the class of concave functions $\varphi:[0,\eta[\,\to[0,+\infty[$ such that $\varphi(0)=0$, $\varphi$ is continuously differentiable on $\,]0,\eta[\,$ and continuous at $0$, and $\varphi'(s)>0$ for every $s\in\,]0,\eta[$.

We say that $\sigma$ has the \emph{Kurdyka--\L{}ojasiewicz (KL) property} at $\overline u\in{\rm dom}\,\partial^L\sigma$ if there exist $\eta\in\,]0,+\infty]$, a neighborhood $U$ of $\overline u$, and $\varphi\in\Phi_\eta$ such that
\[
 \varphi'\bigl(\sigma(u)-\sigma(\overline u)\bigr)\,\operatorname{dist}(0,\partial^L\sigma(u))\geq1
\]
for every $u\in U$ with $\sigma(\overline u)<\sigma(u)<\sigma(\overline u)+\eta$. If $\varphi$ can be taken of the form $\varphi(s)=\dfrac{c}{1-\theta}s^{1-\theta}$ for some $c>0$ and $\theta\in[0,1[$, we say that $\sigma$ satisfies the KL inequality at $\overline u$ with exponent $\theta$; the defining inequality then reduces to
\[
 |\sigma(u)-\sigma(\overline u)|^{\theta}\leq c\,\operatorname{dist}(0,\partial^L\sigma(u))
\]
for every such $u$. A proper lsc function with the KL property at every point of ${\rm dom}\, \partial^L\sigma$ is called a KL function. Real-analytic functions are KL functions, with exponent $\theta \in [1/2,1[$ at every point by the \L{}ojasiewicz gradient inequality; see \cite{ABS} and the references therein.
\end{definition}

For further background on generalized convexity, proximal point methods and the KL property, we refer to \cite{ABS,ADSZ,BC-2,BST,CMA,HKS,Lara-9,LY,P} and references therein.

\section{The Proximal-Gradient Algorithm} \label{sec:03}

Throughout this section, we use the following general assumptions:
\begin{enumerate}
\renewcommand{\theenumi}{H\arabic{enumi}}
\renewcommand{\labelenumi}{(\theenumi)}
 \item\label{hyp:smooth} $g\in C^1(\mathbb R^n)$ is $\mu$-strongly convex with an $L$-Lipschitz gradient, where $L>0$ and $0\leq\mu\leq L$.
 
 \item\label{hyp:regularizer} $h:\mathbb R^n \to \mathbb R\cup\{+\infty\}$ is proper, lower semicontinuous, and  $(\kappa, \gamma)$-strongly quasar-convex with respect to $\overline x_h \in \argmin_{\mathbb R^n}h$, where $0 < \kappa \leq1$ and $\gamma>0$.
\end{enumerate}
We refer to (\ref{hyp:smooth})--(\ref{hyp:regularizer}) together as (H).
Step-size restrictions and additional conditions are stated in each result.
Write
\[
 f=g+h,\qquad F_\kappa=\kappa g+h,\qquad
 \overline x_h \in \argmin\,h, \qquad x^* \in \argmin\,f.
\]
The quadratic-growth inequality guarantees that $\overline x_h$ is unique.
Supporting affine minorants of $g$ and quadratic growth of $h$ make both
objectives coercive, so their minimizers exist. The proximal objective is
also lower semicontinuous and coercive, hence its set of global minimizers
is nonempty and compact. All results hold for every proximal selection.
The iteration below is standard proximal gradient for $F_\kappa$;
$f$ is the comparison objective. The symbol $\mu$ denotes the strong-convexity modulus of $g$
throughout the paper, while $\rho$ denotes the regularization weight in Section~\ref{sec:04}.

\begin{algorithm}[H]
\caption{{Weighted Proximal-Gradient Algorithm with a Strongly Quasar-Convex Regularizer (PG-SQC)}} \label{alg:pg-sqcv}
\begin{description}
 \item[Step 0.] Take $x_{1} \in{\rm dom}\,(h+g)$, set $k=1$, and let $\{\beta_{k}\}_{k\in\mathbb{N}} \subset \mathbb{R}_{++}$.
 \item[Step 1.] Compute
 \begin{align}
  & ~~~ y_{k} = x_{k} - \kappa \beta_{k} \nabla g(x_{k}), \label{step:grad} \\
  & x_{k+1}\in {\rm Prox}_{\beta_{k} h}(y_{k}). \label{step:prox}
 \end{align}
 \item[Step 2.] If $x_{k+1} = x_{k}$,  then \textsc{Stop} and return $x_{k}$. Otherwise, set $k = k+1$ and go to Step 1.
\end{description}
\end{algorithm}

We discuss the stopping criterion below.

\begin{remark}[Exact stopping criterion]\label{rem:stop}
 If the stopping condition in Step~2 is satisfied at some iteration, i.e., $x_{k+1}=x_k$ for the step size $\beta_k>0$ used there, then $x_k$ is already a $\kappa$-critical point of $f=g+h$ in the sense of \eqref{crit:point}. Indeed, proximal optimality of
$x_{k+1}\in{\rm Prox}_{\beta_k h}(y_k)$ gives
\[
 0\in\partial^{L} h(x_{k+1})+\frac{1}{\beta_k}(x_{k+1}-y_k).
\]
Substituting $x_{k+1}=x_k$ and $y_k=x_k-\kappa \beta_k \nabla g(x_k)$ yields $\dfrac{x_k-y_k}{\beta_k} = \kappa\nabla g(x_k)$, hence
\[
 0 \in \kappa \nabla g(x_k)+\partial^{L} h(x_k),
\]
which is exactly \eqref{crit:point} at $\overline x=x_k$. Thus, whenever Step~2 is satisfied, the returned iterate is a critical point of $F_\kappa$, regardless of the admissible step-size schedule $\{\beta_k\}_k$. This is consistent with the convention, used throughout Sections \ref{sec:03}--\ref{sec:04}, of extending the sequence constantly once the algorithm stops.
\end{remark}

As a first result, we present {a proximal comparison inequality valid at arbitrary reference points}.

\begin{theorem} \label{lemma1}
 {Assume (\ref{hyp:regularizer}), let $\beta_k>0$ and $y_k\in\mathbb R^n$, and choose $x_{k+1}\in{\rm Prox}_{\beta_k h}(y_k)$. Then, for every $\delta>0$ and every $x\in{\rm dom}\,h$, we have}
 \begin{equation}\label{ineq3}
 \begin{aligned}
  2\beta_k\kappa\big(h(x_{k+1})-h(x)\big)
  &\leq\|y_k-x\|^2
       -\left(1+\frac{\beta_k\kappa^2\gamma}{2+\delta}\right)
        \|x_{k+1}-x\|^2\\
  &\quad +C_k\|x_{k+1}-y_k\|^2,
 \end{aligned}
\end{equation}
where
$$ C_k :=  \frac{ (2-\kappa) \big(2 + \delta - \kappa \big)}{\beta_k \gamma \kappa^2 \delta}-1.$$
\end{theorem}

\begin{proof}
 By proximal optimality, for every $z\in {\rm dom}\,h$,
\begin{equation}\label{opti}
h(x_{k+1})+\frac{1}{2\beta_k}\| x_{k+1}-y_k\|^2
\leq h(z)+\frac{1}{2\beta_k}\| z-y_k\|^2.
\end{equation}
 Fix $t \in \, ]0,1]$ and set $z_t = (1-t)x_{k+1} + t \overline{x}_{h}$. {Since $h$ is $(\kappa,\gamma)$-strongly quasar-convex}, we have
\begin{equation}\label{qsc1}
 h(z_t) \le h(x_{k+1})-\kappa t\big( h(x_{k+1})-h(\overline{x}_{h}) \big) - t\Big( 1 - \tfrac{t}{2-\kappa} \Big) \frac{\kappa \gamma}{2} \| x_{k+1} - \overline{x}_{h} \|^2.
\end{equation}
Using \eqref{iden:1}, we obtain 
\begin{equation}\label{norm1}
 \| z_t - y_k \|^2 = t \| \overline{x}_{h} - y_k \|^2 + (1-t) \| x_{k+1} - y_k \|^2 - t(1-t) \| x_{k+1} - \overline{x}_{h} \|^2.
\end{equation}
 Taking $z = z_{t}$ in \eqref{opti}, and using \eqref{qsc1} and \eqref{norm1}, we have
\begin{align*}
 2t \beta_k \kappa \big( h(x_{k+1}) - h(\overline{x}_{h}) \big) \leq & ~ t \| y_k - \overline{x}_{h} \|^2 -t \| x_{k+1} - y_{k} \|^2 -t(1-t) \| x_{k+1} - \overline{x}_{h} \|^2 \\ 
 & - {2\beta_k}t \Big( 1-\tfrac{t}{2- \kappa} \Big) \frac{\kappa \gamma}{2} \| x_{k+1} - \overline{x}_{h} \|^2.
\end{align*}
Dividing by $t>0$ and letting $t\downarrow 0$ yields
$$ 2\beta_k \kappa \big( h(x_{k+1}) - h(\overline{x}_{h}) \big) \leq \| y_{k} - \overline{x}_{h} \|^2 - \| x_{k+1}- y_k \|^2 - \| x_{k+1} - \overline{x}_{h} \|^2 - \beta_{k} \kappa \gamma \| x_{k+1} - \overline{x}_{h} \|^2. $$
Adding and subtracting $x$ inside the two squared norms gives
\begin{align*}
 & \hspace{0.4cm} \| y_k-\overline{x}_{h} \|^2 = \| y_k - x \|^2 + 2 \langle y_k-x, x-\overline{x}_{h} \rangle + \| x-\overline{x}_{h} \|^2, \\ 
 &  \| x_{k+1}-\overline{x}_{h} \|^2 = \| x_{k+1}-x\|^2 + 2 \langle x_{k+1}-x, x-\overline{x}_{h} \rangle + \| x- \overline{x}_{h} \|^2.
\end{align*}
Then:
\begin{equation}\label{ineqq}
\begin{aligned}
 2 \beta_k \kappa \big( h(x_{k+1}) - h(\overline{x}_{h}) \big) \leq \| x-y_k \|^2-\| x_{k+1} - x \|^2-\| x_{k+1} -y_k \|^2 \\ \quad 
 + 2 \langle y_{k} - x_{k+1}, x - \overline{x}_{h} \rangle - \kappa \gamma \beta_k \| x_{k+1} - \overline{x}_{h} \|^2.
\end{aligned}
\end{equation}
By \eqref{qwc:sconvex}, we have
 \begin{equation}\label{QG}
  2 \beta_k \kappa \big( h(\overline{x}_{h}) - h(x) \big) \leq - \dfrac{\kappa^2 \gamma \beta_{k}}{2 - \kappa} \| x - \overline{x}_{h} \|^2.
 \end{equation} 
Adding \eqref{ineqq} and \eqref{QG} gives
\begin{align}\label{ineq2}
 2 \beta_k \kappa \big( & \, h(x_{k+1}) - h(x) \big)
 \le \| x-y_k \|^2-\| x_{k+1} - y_k \|^2 - \| x_{k+1} - x \|^2 \\
 &\quad +2\langle y_k - x_{k+1}, x - \overline{x}_{h} \rangle - \kappa \gamma \beta_k \, \| x_{k+1} - \overline{x}_{h} \|^2 - \dfrac{\kappa^2 \gamma \beta_{k}}{2-\kappa} \| x - \overline{x}_{h} \| ^2. \nonumber
\end{align}

Applying identity~\eqref{iden:1} with $u = x - \overline{x}_{h}$, $v = x_{k+1} - \overline{x}_{h}$, and any $\theta \in \, ]0, 1[$, we obtain
\begin{align*}
 0 \leq \| \theta (x-\overline{x}_{h}) + (1-\theta)(x_{k+1} - \overline{x}_{h}) \|^{2} = & ~ (1-\theta) \| x_{k+1} - \overline{x}_{h} \|^{2} + \theta \| x - \overline{x}_{h} \|^{2} \\
 & - \theta (1-\theta) \| x_{k+1} - x \|^{2}. 
\end{align*} 
Rearranging terms gives
\begin{equation}\label{ineq1}
 - \| x_{k+1} - \overline{x}_{h} \|^{2} \leq - \theta \| x_{k+1} - x\|^{2} + \frac{\theta}{1-\theta} \| x - \overline{x}_{h} \|^{2}.
\end{equation}
Next, for any $\varepsilon_k>0$, the following inequality holds
\begin{equation}\label{MA-MG}
 2 \langle y_k - x_{k+1},\, x - \overline{x}_{h} \rangle \le
\frac{1}{\varepsilon_k}\,\| x_{k+1} - y_k \|^2 + \varepsilon_k\,\| x - \overline{x}_{h} \|^2.
\end{equation}
Substituting \eqref{ineq1} and \eqref{MA-MG} into \eqref{ineq2}, we obtain
$$
\begin{aligned}
 2 \beta_k \kappa & ( h(x_{k+1}) - h(x)) \le \| x-y_k \|^2 - ( 1 + \beta_k \kappa\gamma \theta) \| x_{k+1}-x \|^2 \\ 
 & + \Big( \frac{1}{\varepsilon_k} - 1 \Big) \| x_{k+1} - y_k \|^2 + \Big( \varepsilon_k + \frac{\beta_k \kappa \gamma \theta}{1 - \theta} -\frac{\beta_k \kappa^2 \gamma}{2-\kappa} \Big) \| x - \overline{x}_{h} \|^2.
\end{aligned}
$$
We now choose
$$
\varepsilon_k
=\beta_k\,\gamma\!\left(
\frac{\kappa^2}{2-\kappa}
-\frac{\kappa\theta}{1-\theta}
\right)
,
$$
which cancels the coefficient of $\| x-\overline{x}_{h} \|^2$. The parameter $\theta$ can be chosen in $]0, \frac{\kappa}{2}[$ so that $\varepsilon_k > 0$. In particular, for every $\delta > 0$, the choice $\theta = \frac{\kappa}{2 + \delta}$ satisfies this condition. With this choice of parameters, we obtain
$$ 2\beta_k \kappa ( h(x_{k+1}) - h(x)) \le \| x-y_k \|^2 -\big( 1 + \beta_k \kappa \gamma \theta \big) \| x_{k+1} - x \|^2 +\Big( \frac{1}{\varepsilon_k} - 1 \Big) \| x_{k+1} - y_k \|^2. $$
Substituting $\theta = \kappa/(2+\delta)$ and the corresponding value of $\varepsilon_{k}$ into the above inequality yields \eqref{ineq3}.
\end{proof}

We next observe that the coefficient $C_k$ in~\eqref{ineq3} becomes nonpositive whenever
$$
\beta_k \ge \frac{(2-\kappa)\big(2+\delta-\kappa\big)}{\kappa^2\gamma\,\delta}.
$$
In this case, the last term in~\eqref{ineq3} can be omitted. 
For simplicity, we consider $\delta = 1$, leading to the following corollary.

\begin{corollary}\label{coro1}
Under the assumptions of Theorem \ref{lemma1}, if
\begin{equation}\label{param:cond1}
 \beta_k \ge \beta' \ge \frac{(2-\kappa)(3-\kappa)}{\,\kappa^{2}\gamma} > 0, ~ \forall ~ k \in \mathbb{N},
\end{equation}
then 
\begin{equation}\label{eq:to:conv}
 2\beta_k \kappa ( h(x_{k+1})-h(x)) \le \|x-y_k\|^2 - \Big( 1 + \tfrac{\beta_k \kappa^2 \gamma}{3} \Big) \|x_{k+1}-x\|^2.
\end{equation}
\end{corollary}
{
\begin{remark}
Condition~\eqref{param:cond1} is sufficient only for dropping the proximal
residual in this corollary. It is not imposed in the convergence results
below, which retain the residual in Proposition~\ref{propsumm}.
\end{remark}
}

{The main consequence of the previous theorem is the following composite inequality.}

\begin{proposition} {\rm {\bf (Composite Inequality) }} \label{propsumm}
 {Assume (\ref{hyp:smooth})--(\ref{hyp:regularizer}), let $f=g+h$, and let $\{x_k\}_{k}$ be generated by Algorithm~\ref{alg:pg-sqcv}. For every $\delta>0$} and any $x \in {\rm dom}\,f$, the following ine\-qua\-li\-ty holds:
\begin{equation}\label{compositeineq}
 \begin{aligned}
  2 \kappa \beta_k \big( f(x_{k+1})-f(x) \big) & \leq (1-\kappa\beta_k\mu)\| x-x_k \|^2 - \left(1 + \frac{\beta_k \kappa^2 \gamma}{2+\delta} \right) \| x_{k+1}-x \|^2 \\
  & + D_k \| x_{k+1}-y_k \|^2 - (1 - \kappa \beta_k L) \| x_{k+1} - x_k \|^2,
\end{aligned}
\end{equation}
where $D_k :=  \frac{ (2-\kappa) \big(2 + \delta - \kappa \big)}{\beta_k \gamma \kappa^2 \delta}$.
\end{proposition}

\begin{proof}
By the $\mu$-strong convexity of $g$, for all $x\in {\rm dom}\,h$, we have
\begin{equation}\label{eqconvex}
 g(x_k) + \langle \nabla g(x_k), x-x_{k} \rangle + \frac{\mu}{2} \|x-x_k\|^2 \leq  g(x).
\end{equation}
Applying the descent lemma \eqref{descent:lemma} yields
\begin{equation}\label{descentlemma}
 g(x_{k+1}) \le g(x_k) + \langle \nabla g(x_k), x_{k+1} - x_k \rangle + \frac{L}{2}\| x_{k+1} - x_k \|^2.
\end{equation}
Subtracting \eqref{eqconvex} from \eqref{descentlemma}, we obtain
\begin{equation}\label{eq004}
 g(x_{k+1}) - g(x) \le \langle \nabla g(x_k), x_{k+1} - x \rangle + \frac{L}{2}\| x_{k+1} - x_k\|^2 - \frac{\mu}{2}\|x-x_k\|^2. 
\end{equation}
Observe that 
\begin{equation}\label{eq5}
\begin{aligned}
    \| x-y_k\| ^2 &=  \| x-x_k\| ^2+2\langle x-x_k, x_k-y_k\rangle+\| x_k-y_k\| ^2 \\
    &= \| x-x_k\| ^2 +2\beta_k \kappa \langle x-x_k, \nabla g(x_k)\rangle +\| x_k-y_k\| ^2.
\end{aligned}
\end{equation}
{Expanding the inner product and using \eqref{eq004} gives} 
\begin{equation}\label{eq003}
 \begin{aligned}
 2\beta_k \kappa & \langle x-x_k, \nabla g(x_k) \rangle =2\beta_k \kappa \langle x-x_{k+1}, \nabla g(x_k)\rangle +2\beta_k \kappa \langle x_{k+1}-x_k, \nabla g(x_k) \rangle \\
 & \leq 2\kappa \beta_k \big(g(x)-g(x_{k+1})+ \tfrac{L}{2} \| x_{k+1}-x_{k}\| ^2 - \tfrac{ \mu}{2} \|x-x_k\|^2\big)  \\
  & ~~~ + 2\langle x_{k+1}-x_{k}, x_k-y_k \rangle.
 \end{aligned}
 \end{equation}
Using identity \eqref{3:points},
\begin{equation} 
 2 \langle x_{k+1}-x_{k}, x_k-y_k \rangle = \| x_{k+1}-y_k \|^2 - \| x_{k+1} - x_k \|^2 - \| y_k - x_k \|^2, \notag
\end{equation}
we obtain
 \begin{equation}\label{eq10}
 \begin{aligned}
  2\beta_k \kappa \langle x-x_k, \nabla g(x_k) \rangle & \leq  2\kappa \beta_k \big( g(x) - g(x_{k+1})\big) - \kappa\beta_k\mu\|x-x_k\|^2 + \| x_{k+1} - y_k \|^2 \\ 
  & - (1-\kappa \beta_k L) \| x_{k+1} - x_k \|^2 -\| y_k - x_k \|^2.
 \end{aligned}
\end{equation}
Substituting \eqref{eq10} into \eqref{eq5} yields
\begin{align}
 \| x - y_k \|^2 \le & ~ (1-\kappa\beta_k\mu)\| x_k-x \| ^2 + 2\kappa \beta_k \big( g(x) - g(x_{k+1}) \big) + \| x_{k+1} - y_k \| ^2 \notag \\
 & - (1 - \kappa \beta_k L) \| x_{k+1} - x_k \| ^2. \label{eq6}
\end{align}
Rewriting \eqref{eq6}, we get
\begin{align}
 2 \kappa \beta_k (g(x_{k+1}) - g(x)) \le & ~ (1-\kappa\beta_k\mu)\| x-x_k\| ^2 - \| x-y_k\| ^2 + \| x_{k+1} -y_k\|^2 \notag \\
 & - (1 - \kappa \beta_k L) \| x_{k+1} - x_k \|^2. \label{eq7}
\end{align}
Adding \eqref{eq7} to \eqref{ineq3} from Theorem~\ref{lemma1} completes the proof.
\end{proof}

\subsection{Convergence Results}

{Throughout this subsection, $\{x_k\}_{k}$ is generated by Algorithm \ref{alg:pg-sqcv} from $x_1 \in {\rm dom}\,h$. If the algorithm stops, its sequence is extended constantly. The general hypotheses are (\ref{hyp:smooth})--(\ref{hyp:regularizer}). In particular, quasar-convexity is global and $\mu \ge 0$ denotes the strong-convexity modulus of $g$.}

\begin{lemma} \label{lem:qg-quasar}
{Under (\ref{hyp:smooth})--(\ref{hyp:regularizer})}, for every $x\in{\rm dom}\,h$, we have
\[
h(x) \ge h(\overline{x}_{h}) + \frac{\kappa\gamma}{2(2-\kappa)}\|x-\overline{x}_{h}\|^2.
\]
Consequently, $F_\kappa:=\kappa g+h$ {and $f=g+h$ are} coercive, and in particular {have} bounded sublevel sets {and attain their minima}.
\end{lemma}

\begin{proof}
The first claim is \eqref{qwc:sconvex}. Convexity gives $g(x)\ge g(\overline{x}_{h})+\langle\nabla g(\overline{x}_{h}),x-\overline{x}_{h}\rangle$. The quadratic lower bound for $h$ dominates this affine term as $\|x\|\to+\infty$, proving coercivity of both $F_\kappa$ and $f$. Properness and lower semicontinuity then imply that both objectives attain their minima.
\end{proof}

\begin{lemma} \label{lem:kappa-decrease}
 {Assume (\ref{hyp:smooth})--(\ref{hyp:regularizer}).} For any $\kappa \in \, ]0,1]$ and $0 < \beta_k \le 1/(\kappa L)$, we have
\[
 F_\kappa (x_{k+1}) \le F_\kappa (x_k) - \Big( \frac1{2 \beta_k}-\frac{\kappa L}2 \Big) \|x_{k+1}-x_k\|^2, \qquad F_\kappa := \kappa g + h.
\]
\end{lemma}
\begin{proof}
By proximal optimality of $x_{k+1}$ with comparison point $z=x_k$, we have 
$$h(x_{k+1}) + \frac1{2 \beta_k} \|x_{k+1} - y_k\|^2 \le h(x_k) + \frac1{2 \beta_k} \|x_k - y_k\|^2.$$ 
Using $y_k=x_k-\kappa\beta_k\nabla g(x_k)$ to expand both squared norms, we cancel the common terms $\frac{\kappa^2\beta_k}{2}\|\nabla g(x_k)\|^2$ and obtain
\[
 h(x_{k+1}) \le h(x_k) - \kappa \langle \nabla g(x_k), x_{k+1}-x_k\rangle - \tfrac{1}{2 \beta_k}  \|x_{k+1}-x_k\|^2.
\]
Adding $\kappa$ times the descent lemma \eqref{descent:lemma}, namely $\kappa g(x_{k+1}) \le \kappa g(x_k)+ \kappa\langle\nabla g(x_k), x_{k+1}-x_k \rangle + \frac{\kappa L}2\|x_{k+1}-x_k\|^2$, cancels the inner-product terms and proves the claim.
\end{proof}

\begin{lemma} \label{lem:bounded}
 {Assume (\ref{hyp:smooth})--(\ref{hyp:regularizer}).} Let $0<\beta_k \equiv \beta < \dfrac{1}{\kappa L}$. Then, by Lemma \ref{lem:kappa-decrease}, $\{F_\kappa (x_k)\}_{k}$ is non\-in\-crea\-sing, and
\begin{enumerate}
 \item[$(a)$] $\{x_k\}_{k}$ is bounded, and consequently $M := \sup_{k \in \mathbb N} \|\nabla g(x_k)\| < + \infty$;
 
 \item[$(b)$] $\displaystyle \sum_{k=1}^{\infty} \|x_{k+1}- x_k\|^2 < + \infty$. In particular, $\|x_{k+1}- x_k\|\to0$.
\end{enumerate}
\end{lemma}
\begin{proof}
 By Lemma \ref{lem:kappa-decrease}, $F_\kappa (x_{k+1}) \le F_\kappa (x_k) - \big( \tfrac1{2 \beta} -\tfrac{\kappa L}2 \big) \|x_{k+1} - x_k\|^2$, so $\{F_\kappa (x_k)\}_{k}$ is nonincreasing and $F_\kappa (x_k) \le F_\kappa(x_1)$ for every $k$.

 $(a)$: By Lemma~\ref{lem:qg-quasar}, $F_\kappa$ is coercive, so the sublevel set $\{x: \, F_\kappa(x) \le F_\kappa(x_1)\}$ is bounded. Since $\{x_k\}_{k}$ lies in this set, it is bounded, and by continuity of $\nabla g$, we have $M := \sup_k \|\nabla g(x_k)\| < + \infty$.

 $(b)$: Since $\beta < \frac{1}{(\kappa L)}$, $c_0 :=  \tfrac1{2 \beta} - \tfrac{\kappa L}2 > 0$. Summing the inequality of Lemma \ref{lem:kappa-decrease} from $k=1$ to $K$ gives
\[
 c_0 \sum_{k=1}^{K}\|x_{k+1}-x_k\|^2 \;\le\; F_\kappa (x_1) - F_\kappa (x_{K+1}) \;\le\; F_\kappa(x_1)-\inf_{\R^n} F_\kappa,
\]
and the right-hand side is finite since $F_\kappa$ is {proper, lsc and} coercive. Letting $K \to +\infty$ gives $\sum_k \|x_{k+1}-x_k \|^2 < +\infty$, hence $\|x_{k+1}-x_k\| \to 0$.
\end{proof}

We now show that every accumulation point of the sequence generated by Algorithm \ref{alg:pg-sqcv} is a $\kappa$-critical point of $f = g+h$ in the sense of \eqref{crit:point}, i.e., a critical point of $F_\kappa = \kappa g + h$.

\begin{lemma}[Subgradient bound] \label{lem:subgrad-bound}
 {Assume (\ref{hyp:smooth})--(\ref{hyp:regularizer}).} Let $0<\beta_k \equiv \beta < 1/(\kappa L)$. For every $k \in \mathbb{N}$, define
 \[
  w_{k+1} := \kappa \big( \nabla g(x_{k+1}) - \nabla g(x_k) \big) + \frac{x_k - x_{k+1}}{\beta}.
 \]
 Then $w_{k+1} \in \kappa \nabla g(x_{k+1}) + \partial^{L} h(x_{k+1})$ and
 \[
  \| w_{k+1} \| \;\le\; \Big( \kappa L + \frac{1}{\beta} \Big) \| x_{k+1} - x_k \|, ~ \forall ~ k \in \mathbb{N}.
 \]
 In particular, by Lemma \ref{lem:bounded}$(b)$, $w_{k+1} \to 0$ as $k \to +\infty$.
\end{lemma}

\begin{proof}
 By proximal optimality of {$x_{k+1} \in {\rm Prox}_{\beta h}(y_k)$} and the sum rule for $\partial^{L} h$, we have
 \[
  0 \in \partial^{L} h(x_{k+1}) + \frac{1}{\beta}(x_{k+1} - y_k) ~ \Longleftrightarrow ~ \frac{y_k - x_{k+1}}{\beta} \in \partial^{L} h(x_{k+1}).
 \]
 Since $y_k = x_k - \kappa \beta \nabla g(x_k)$,
 \[
  \frac{y_k - x_{k+1}}{\beta} = \frac{x_k - x_{k+1}}{\beta} - \kappa \nabla g(x_k) \in \partial^{L} h(x_{k+1}).
 \]
 Adding $\kappa \nabla g(x_{k+1})$ to both sides gives
 \[
  w_{k+1} = \kappa \big( \nabla g(x_{k+1}) - \nabla g(x_k) \big) + \frac{x_k - x_{k+1}}{\beta} \in \kappa \nabla g(x_{k+1}) + \partial^{L} h(x_{k+1}).
 \]
 By the triangle inequality and the $L$-Lipschitz continuity \eqref{L:smooth} of $\nabla g$, we obtain
 \[
  \| w_{k+1} \| \le \kappa \| \nabla g(x_{k+1}) - \nabla g(x_k) \| + \frac{1}{\beta} \| x_{k+1} - x_k \| \le \Big( \kappa L + \frac{1}{\beta} \Big) \| x_{k+1} - x_k \|,
 \]
 which is the stated bound. The convergence $w_{k+1}\to0$ then follows from $\|x_{k+1}-x_k\|\to0$ (Lemma \ref{lem:bounded}$(b)$).
\end{proof}

\begin{theorem}[Subsequential convergence to critical points] \label{thm:critical}
 {Assume (\ref{hyp:smooth})--(\ref{hyp:regularizer}).} Let $0<\beta_k \equiv \beta < 1/(\kappa L)$. Then $\{x_k\}_k$ admits at least one accumulation point, and every accumulation point $\widehat{x}$ of $\{x_k\}_k$ is a $\kappa$-critical point of $f=g+h$, i.e.,
 \[
  0 \in \kappa \nabla g(\widehat{x}) + \partial^{L} h(\widehat{x}).
 \]
\end{theorem}

\begin{proof}
 By Lemma \ref{lem:bounded}$(a)$, $\{x_k\}_k$ is bounded, so it admits a convergent subsequence $x_{k_j} \rightarrow \widehat{x}$. By Lemma \ref{lem:bounded}$(b)$, $x_{k_j+1}-x_{k_j}\to0$, hence $x_{k_j+1}\rightarrow \widehat{x}$ as well.
 {First we verify convergence of the function values, which is required for the limiting-subdifferential argument. By Lemma~\ref{lem:kappa-decrease}, the sequence $\{F_\kappa(x_k)\}_{k}$ is nonincreasing. By Lemma~\ref{lem:qg-quasar}, $F_\kappa$ has a finite minimum, so this sequence is bounded below. Therefore, $F_\kappa(x_k)\to\ell:=\inf_{k\geq1}F_\kappa(x_k)\in\R$. Lower semicontinuity gives $F_\kappa(\widehat{x})\le\ell$, so $h(\widehat{x})< + \infty$. Proximal optimality with comparison point $\widehat{x}$ gives}
 {
\[
 h(x_{k_j+1})+\frac{1}{2\beta}\|x_{k_j+1}-y_{k_j}\|^2
 \le h(\widehat{x})+\frac{1}{2\beta}\|\widehat{x}-y_{k_j}\|^2.
 \]
}
 {Both squared norms converge to $\kappa^2\beta^2\|\nabla g(\widehat{x})\|^2$. Together with lower semicontinuity, this proves $h(x_{k_j+1})\to h(\widehat{x})$, and hence $F_\kappa(\widehat{x})=\ell$.}
 By Lemma \ref{lem:subgrad-bound}, $w_{k_j+1} \in \kappa \nabla g(x_{k_j+1}) + \partial^{L} h(x_{k_j+1})$ with $w_{k_j+1}\rightarrow 0$. Since $\kappa\nabla g(\cdot)$ is continuous and {the limiting subdifferential} is closed {along value-convergent sequences}, letting $j \rightarrow +\infty$ yields
 \[
  0 = \lim_{j\to+\infty} w_{k_j+1} \;\in\; \kappa \nabla g(\widehat{x}) + \partial^{L} h(\widehat{x}),
 \]
 which is \eqref{crit:point}.
\end{proof}

\begin{remark}
 Theorem \ref{thm:critical} {uses $\gamma>0$ from (\ref{hyp:regularizer})} (needed only to guarantee boundedness of $\{x_k\}_k$ via Lemma \ref{lem:bounded}$(a)$); it requires no additional condition on $\mu$. When $\kappa=1$, \eqref{crit:point} is precisely stationarity of $f=g+h$ itself. For $\kappa<1$, however, a $\kappa$-critical point of $f$ need not be a critical point of $f$, since the {iteration decreases $F_\kappa=\kappa g+h$}. {This distinction is illustrated by the nonconvex regularizer in Example~\ref{ex:weighting-bias}.}
\end{remark}

{
\begin{corollary}[Convergence under the KL property] \label{cor:kl}
Under the hypotheses of Theorem~\ref{thm:critical}, if $F_\kappa$ has the Kurdyka--\L{}ojasiewicz (KL) property at every accumulation point, then $\sum_k\|x_{k+1}-x_k\|<\infty$ and the whole sequence converges to a critical point of $F_\kappa$.
\end{corollary}
\begin{proof}
Lemma~\ref{lem:kappa-decrease}, Lemma~\ref{lem:subgrad-bound}, and the value-convergent subsequence constructed in Theorem~\ref{thm:critical} verify the sufficient-decrease, relative-error and continuity conditions of the standard KL convergence theorem \cite[Theorem~2.9]{ABS}.
\end{proof}
}

\begin{lemma}[Radial subgradient inequality]\label{lem:radial-subgradient}
Under (\ref{hyp:regularizer}), for every $u\in{\rm dom}\,h$ and $v\in\partial^Lh(u)$, we have
\begin{equation}\label{eq:radial-subgradient}
 \langle v,u-\overline{x}_h\rangle\ge\kappa\big(h(u)-h(\overline{x}_h)\big)+\frac{\kappa\gamma}{2}\|u-\overline{x}_h\|^2
 \ge\frac{\kappa\gamma}{2-\kappa}\|u-\overline{x}_h\|^2.
\end{equation}
\end{lemma}
\begin{proof}
If $u=\overline{x}_h$, the assertion is immediate. Otherwise, for a Fr\'echet subgradient, its defining lower estimate along $(1-t)u+t\overline{x}_h$ and the quasar-convex upper estimate give the first inequality after division by $t$ and passage to $t\downarrow0$. For a limiting subgradient, pass to the limit along $u_j\to u$, $h(u_j)\to h(u)$ and Fr\'echet subgradients $v_j\to v$. The second inequality follows from \eqref{qwc:sconvex}.
\end{proof}

The radial inequality also yields a comparison valid at every critical point of $F_\kappa$.
\begin{proposition}[Comparison at critical points]\label{prop:critical-comparison}
Assume (\ref{hyp:smooth})--(\ref{hyp:regularizer}). Put
\[
 C_\mu=\frac{(2-\kappa)+\sqrt{2(2-\kappa)(1+\mu/\gamma)}}{\kappa\gamma+2\mu}.
\]
For every critical point $u$ of $F_\kappa$ and every $x^*\in\argmin_{\R^n}f$,
\[
 \|u-x^*\|\le C_\mu\|\nabla g(u)\|.
\]
\end{proposition}

\begin{proof}
 Write $z=u-x^*$, $t=x^*-\overline{x}_h$ and $p=\nabla g(u)$. Since $-\kappa p\in \partial^Lh(u)$, Lemma~\ref{lem:radial-subgradient}, quadratic growth at $x^*$, and strong convexity of $g$ give
\begin{align*}
 0\le f(u)-f(x^*)&\le-\langle p,t\rangle-\frac\gamma2\|z+t\|^2-\frac{\kappa\gamma}{2(2-\kappa)}\|t\|^2-\frac\mu2\|z\|^2\\
 &=-\frac{\gamma+\mu}{2}\|z\|^2-\langle\gamma z+p,t\rangle-\frac\gamma{2-\kappa}\|t\|^2\\
 &\le-\frac{\gamma+\mu}{2}\|z\|^2+\frac{2-\kappa}{4\gamma}\|\gamma z+p\|^2.
\end{align*}
Completing the square in $t$ gives the last step. Expansion and Cauchy--Schwarz yield $(\kappa\gamma+2\mu)\|z\|^2\le2(2-\kappa)\|z\|\|p\|+((2-\kappa)/\gamma)\|p\|^2$. The nonnegative root of the associated quadratic equation is $C_\mu\|p\|$.
\end{proof}

\begin{theorem}[Geometric and asymptotic distances to a global minimizer]\label{thm:exact}
Assume (\ref{hyp:smooth})--(\ref{hyp:regularizer}), let
$0<\beta_k\equiv\beta<1/(\kappa L)$, and fix
$x^*\in\argmin_{\mathbb R^n}f$. Put
\[
 G^*=\|\nabla g(x^*)\|,\qquad
 M_\infty=\limsup_k\|\nabla g(x_k)\|,
\]
\[
 q=\frac{1-\kappa\beta\mu}{1+\beta\kappa\gamma/(2-\kappa)}\in\,]0,1[,
 \qquad
 R_f=\frac{(2-\kappa)G^*}{\kappa\gamma}
 \left(1+\frac{\kappa\gamma+(2-\kappa)L}{\gamma+(2-\kappa)\mu}\right).
\]
Let $C_\mu$ be the constant in Proposition~\ref{prop:critical-comparison}.
Then:
\begin{enumerate}
\item[$(a)$] For every $k\geq1$,
\begin{equation}\label{eq:geometric-neighborhood}
 \|x_k-x^*\|
 \leq q^{k-1}\left(\|x_1-x^*\|+\frac{2(2-\kappa)G^*}{\kappa\gamma}\right)
       +(1-q^{k-1})R_f,
\end{equation}
and consequently
\begin{equation}\label{eq:geometric-distance-global}
 \operatorname{dist}\bigl(x_k,\overline{\mathbb B}(x^*,R_f)\bigr)
 \leq q^{k-1}
       \left(\|x_1-x^*\|+\frac{2(2-\kappa)G^*}{\kappa\gamma}-R_f\right)_+.
\end{equation}
The radius $R_f$ is independent of $\beta$.
\item[$(b)$] The asymptotic distance satisfies
\begin{equation}\label{conv:ball}
 \limsup_k\|x_k-x^*\|
 \leq \min\{R_f,\ C_\mu M_\infty\}.
\end{equation}
\end{enumerate}
If $G^*=0$, then $R_f=0$ and
$\|x_k-x^*\|\leq q^{k-1}\|x_1-x^*\|$; in particular, the
sequence converges globally and linearly to $x^*$.
\end{theorem}

\begin{proof}
\noindent$(a)$ If $u\in{\rm Prox}_{\beta h}(y)$, proximal optimality and
Lemma~\ref{lem:radial-subgradient} give
\[
 \langle y-\overline x_h,u-\overline x_h\rangle
 \geq\left(1+\frac{\beta\kappa\gamma}{2-\kappa}\right)
       \|u-\overline x_h\|^2.
\]
By Cauchy--Schwarz,
\[
 \left(1+\frac{\beta\kappa\gamma}{2-\kappa}\right)
 \|u-\overline x_h\|\leq\|y-\overline x_h\|.
\]
The gradient map $I-\kappa\beta\nabla g$ is
$(1-\kappa\beta\mu)$-Lipschitz. Indeed, if $\mu<L$, convex cocoercivity
applied to $g-\mu\|\cdot\|^2/2$, whose quadratic upper model has
modulus $L-\mu$, gives, with $s=\nabla g(x)-\nabla g(z)$,
\[
 (L+\mu)\langle x-z,s\rangle
 \geq\|s\|^2+\mu L\|x-z\|^2.
\]
Together with $\|s\|\geq\mu\|x-z\|$ and
$\kappa^2\beta^2-2\kappa\beta/(L+\mu)\leq0$, this yields
\begin{align*}
 \|x-z-\kappa\beta s\|^2
 &\leq\left(1-\frac{2\kappa\beta\mu L}{L+\mu}\right)\|x-z\|^2
       +\left(\kappa^2\beta^2-\frac{2\kappa\beta}{L+\mu}\right)\|s\|^2\\
 &\leq(1-\kappa\beta\mu)^2\|x-z\|^2.
\end{align*}
If $\mu=L$, the shifted function is affine and the same
conclusion follows directly.

Set $d_k=\|x_k-\overline x_h\|$.
The two preceding estimates imply
\[
 d_{k+1}\leq qd_k+
 \frac{\kappa\beta\|\nabla g(\overline x_h)\|}
      {1+\beta\kappa\gamma/(2-\kappa)}.
\]
Iterating this recurrence and using the definition of $q$ gives
\[
 d_k\leq q^{k-1}d_1+
       (1-q^{k-1})\frac{\|\nabla g(\overline x_h)\|}
                            {\mu+\gamma/(2-\kappa)}.
\]
Fermat's rule at $x^*$ gives
$-\nabla g(x^*)\in\partial^Lh(x^*)$.
The radial inequality therefore implies
\[
 \frac{\kappa\gamma}{2-\kappa}\|x^*-\overline x_h\|^2
 \leq\langle-\nabla g(x^*),x^*-\overline x_h\rangle
 \leq G^*\|x^*-\overline x_h\|,
\]
so $\|x^*-\overline x_h\|\leq(2-\kappa)G^*/(\kappa\gamma)$.
Lipschitz continuity and the definition of $R_f$ also give
\[
 \frac{\|\nabla g(\overline x_h)\|}{\mu+\gamma/(2-\kappa)}
 \leq R_f-\frac{(2-\kappa)G^*}{\kappa\gamma}.
\]
Consequently,
\begin{align*}
 \|x_k-x^*\|
 &\leq d_k+\|x^*-\overline x_h\|\\
 &\leq q^{k-1}\left(\|x_1-x^*\|+\frac{(2-\kappa)G^*}{\kappa\gamma}\right)\\
 &\quad +(1-q^{k-1})\left(R_f-\frac{(2-\kappa)G^*}{\kappa\gamma}\right)
       +\frac{(2-\kappa)G^*}{\kappa\gamma}\\
 &=q^{k-1}\left(\|x_1-x^*\|+\frac{2(2-\kappa)G^*}{\kappa\gamma}\right)
       +(1-q^{k-1})R_f.
\end{align*}
This proves \eqref{eq:geometric-neighborhood}; subtracting $R_f$
and taking positive parts proves \eqref{eq:geometric-distance-global}.
In particular, $\limsup_k\|x_k-x^*\|\leq R_f$.

\medskip
\noindent$(b)$ Apply Proposition~\ref{propsumm} at $x=x^*$.
Discard the nonnegative objective gap and the nonpositive step term.
By Lemma~\ref{lem:bounded}, the iterates are bounded and
$x_{k+1}-x_k\to0$; hence
\[
 \limsup_k\|x_{k+1}-y_k\|^2
 =\kappa^2\beta^2M_\infty^2.
\]
Taking upper limits in the composite inequality gives, for every $\delta>0$,
\[
 \limsup_k\|x_k-x^*\|^2\leq K_\delta^2M_\infty^2,
 \qquad
 K_\delta^2=
 \frac{(2-\kappa)(2+\delta-\kappa)(2+\delta)}
 {\kappa\gamma\delta\,[\kappa\gamma+\mu(2+\delta)]}.
\]
The definition of $C_\mu$ implies
$(\kappa\gamma+2\mu)C_\mu^2=2(2-\kappa)C_\mu+(2-\kappa)/\gamma$.
A direct expansion then gives
\begin{equation}\label{eq:delta-optimization}
 K_\delta^2-C_\mu^2
 =\frac{\big((\kappa\gamma C_\mu-(2-\kappa))\delta-2(2-\kappa)\big)^2}
 {2\kappa\gamma\delta\,[\kappa\gamma+\mu(2+\delta)]}.
\end{equation}
If $\mu<\kappa\gamma/(2-\kappa)$, then $\kappa\gamma C_\mu>2-\kappa$ and
the choice
\[
 \delta_* = \frac{2(2-\kappa)}{\kappa\gamma C_\mu-(2-\kappa)}
\]
gives $K_{\delta_*}=C_\mu$.
For $\mu=0$, this reduces to $\delta_*=\sqrt{2(2-\kappa)}$.
At $\mu=\kappa\gamma/(2-\kappa)$, letting $\delta\to+\infty$ in
\eqref{eq:delta-optimization} gives the same bound.

If $\mu>\kappa\gamma/(2-\kappa)$, take a subsequence attaining
$\limsup_k\|x_k-x^*\|$ and, by boundedness, a convergent
subsequence with limit $u$. Theorem~\ref{thm:critical} makes $u$
a critical point of $F_\kappa$. Proposition~\ref{prop:critical-comparison}
and continuity of $\nabla g$ yield
\[
 \limsup_k\|x_k-x^*\|
 =\|u-x^*\|\leq C_\mu\|\nabla g(u)\|
 \leq C_\mu M_\infty.
\]
This establishes the comparison with $C_\mu M_\infty$ for all
$\mu\geq0$. Taking the minimum with the bound $R_f$ from part $(a)$ proves
\eqref{conv:ball}. The assertion for $G^*=0$ follows from part $(a)$.
\end{proof}

The geometric rate in part $(a)$ concerns the ball of radius $R_f$;
part $(b)$ gives the best of the available asymptotic distance bounds.

\begin{example}[Comparison of the two bounds]\label{ex:complementary-bounds}
Let $h(x)=x^2/6$ and $g_a(x)=a(x-1)^2/2$, with
$\kappa=1/2$, $\gamma=\beta=1$, and $a\in\{1/10,1\}$.
These functions satisfy (H), with $\mu=L=a$. The iteration is
\[
 x_{k+1}=\frac34\left(1-\frac a2\right)x_k+\frac{3a}{8},
\]
so it converges to $\widehat x=3a/(2+3a)$, whereas the global
minimizer of $f$ is $x^*=3a/(1+3a)$. Hence
\[
 G^*=\frac{a}{1+3a},\qquad M_\infty=\frac{2a}{2+3a}.
\]
Substitution gives the following values:
\begin{center}
\begin{tabular}{ccc}
\toprule
$a$ & $R_f$ & $C_\mu M_\infty$\\
\midrule
$1/10$ & $0.361204$ & $0.411999$\\
$1$ & $1.350000$ & $0.631918$\\
\bottomrule
\end{tabular}
\end{center}
The ordering reverses between the two choices, so neither bound
dominates the other in general.
\end{example}

{
\begin{example}[Weighting bias with a nonconvex regularizer] \label{ex:weighting-bias}
Let $g(x)=(x-3)^2$, $h(x)=x^2-\tfrac32x^4+x^6$ and $\kappa=\gamma=\beta=\tfrac12$. Here, $L=\mu=2$, and $0<\beta<1/(\kappa L)=1$. Since $h(x)=x^2((x^2-3/4)^2+7/16)$, its unique minimizer is zero. Moreover, $h''(x)=30(x^2-3/10)^2-7/10$, so $h$ is nonconvex. Nevertheless, it is globally $(\kappa, \gamma)$-strongly quasar-convex by \eqref{diff:squasar}, since
\[
 xh'(x)-\kappa h(x)-\frac{\kappa \gamma}{2}x^2 = x^2 \left(\frac{11}{2}(x^2-21/44)^2+\frac{43}{352} \right) \geq 0.
\]
The proximal objective has second derivative at least $13/10$, so each proximal minimizer is unique. The iteration satisfies
\[
 y_k = \frac{x_k+3}{2}, \qquad y_k = x_{k+1} + \tfrac12h'(x_{k+1}).
\]
The derivative of $u\mapsto u+\tfrac12h'(u)$ is at least $13/20$;
thus the update is a contraction with factor $10/13$.
Since $h'(1)=2$, its fixed point is $\widehat x=1$, and every initial point gives
\[
 |x_k-1|\leq(10/13)^{k-1}|x_1-1|.
\]
Finally, $F_\kappa''\geq3/10$, so $\widehat x$ is its unique global minimizer,
whereas
\[
 F_\kappa'(1)=\tfrac12(-4)+2=0,\qquad f'(1)=-4+2=-2\ne0.
\]
This example separates $\kappa$-criticality from ordinary criticality
with a nonconvex regularizer satisfying the global hypotheses.
\end{example}

The general results establish critical-point convergence for $F_\kappa$
and quantitative distance estimates relative to global minimizers of $f$;
they do not assert global optimality of the limit under the general hypotheses.
The common-minimizer case of Theorem~\ref{thm:exact}
ensures global linear convergence to the common global minimizer.
}

\section{Applications to  $\ell_p$-Regularized Least Squares ($0<p<1$)} \label{sec:04}

\subsection{{Model, Quasar-Convexity and Confinement}}\label{sec:lp-model}

We apply the framework of Section~\ref{sec:03} to the classical nonconvex sparse recovery problem
\begin{equation}\label{eq:lp-problem}
 \min_{x\in \R^n}\ F(x) := \|Ax-b \|_2^2 + \rho \|x\|_p^p,
\end{equation}
where $A\in \R^{m\times n}$, $b\in\R^m$, $0<p<1$ and $\rho>0$. {This model arises} in sparse signal/image recovery, LASSO-type estimation{~\cite{KF}}, and compressed sensing; the {$\ell_p$ power penalty} {has been studied as a sparsity-promoting alternative} to $\ell_1$ (see {\cite{Lu,Survey,XCXZ}}).

{Set $g(x)=\|Ax-b\|_2^2$ and $h(x)=\rho\|x\|_p^p$. The quasar-convexity of this penalty on bounded sets was established in \cite[Example~8]{BLL}; we recall its modulus below before applying Algorithm~\ref{alg:pg-sqcv}.}

{Algorithm~\ref{alg:pg-sqcv} decreases} the {weighted objective}
{
\[
 F_\kappa=\kappa g+h
 =\kappa\left(\|Ax-b\|_2^2+\frac\rho\kappa\|x\|_p^p\right).
\]
}
{Thus its effective regularization parameter is $\rho/\kappa$, while $F$ in \eqref{eq:lp-problem} is the comparison objective. If a prescribed regularization parameter $\lambda$ is to be used in the objective optimized} by {the iteration, taking $h=\kappa\lambda\|\cdot\|_p^p$ gives the usual proximal-gradient method} for {that objective with step $\kappa\beta$}.

\begin{lemma} \label{lem:scale}
 If $h$ is $(\kappa, \gamma)$-strongly quasar-convex {on a star-shaped set $K$} with respect to $\overline{x}_{h} \in \argmin_{\R^n}h$, then for every $\rho>0$, $\rho h$ is $(\kappa,\rho \gamma)$-strongly quasar-convex {on $K$} with respect to the same point $\overline{x}_{h}$.
\end{lemma}

\begin{proof}
 Multiplying the inequality in Definition~\ref{def:quasar} by $\rho>0$ preserves its direction. The terms $\rho h(\overline{x}_{h})$, $\rho h(y)$ and $\rho h(\lambda\overline{x}_{h}+(1-\lambda)y)$ replace their unscaled counterparts, while the last term becomes $-\lambda\big(1-\frac{\lambda}{2-\kappa} \big)\frac{\kappa(\rho\gamma)}{2}\|y-\overline{x}_{h}\|^2$. Since $\rho>0$, {$\argmin_{\R^n}\rho h= \argmin_{\R^n}h$}.
\end{proof}

\begin{proposition}\label{prop:lp-quasar-exact} {\rm (see \cite[Example 8]{BLL})}
 Let $0<p<1$ and $h(x) := \|x\|_p^p= \sum_{i=1}^n \lvert x_i \rvert^p$. Then $\overline{x}_{h}=0 \in \argmin_{\R^n}h$, and for every $\kappa \in \, ]0, p[$ and every $R>0$, $h$ is $(\kappa, \gamma)$-strongly quasar-convex with respect to $\overline{x}_{h}=0$ on $K = \overline{B}(0,R)$, with the sharp modulus
\[
 \gamma(\kappa,p,R) = \frac{2(p-\kappa)}{\kappa} R^{p-2}.
\]
\end{proposition}

{The stated modulus follows from \cite[Example~8]{BLL} by rescaling the unit ball to radius $R$. Sharpness follows by evaluating the defining inequality at $x=(R,0,\ldots,0)$, dividing by $\lambda>0$, and letting $\lambda\downarrow0$.}

\begin{remark} \label{rem:confinement}
 Every result invoking Proposition \ref{prop:lp-quasar-exact} is stated relative to $K=\overline{B}(0,R)$, since the quasar-convexity modulus $\gamma(\kappa,p,R)\to0$ as $R\to+\infty$. The penalty is not strongly quasar-convex on the whole space, because $\|y\|_p^p/\|y\|_2^2\to0$ as $\|y\|_2\to+\infty$. Applying Theorem~\ref{thm:exact} therefore requires {the iterates and comparison points to belong to} the same {ball} used to compute $\gamma$. {The following proposition supplies such a ball directly} from {the weighted objective}, {without using} $\gamma$.
\end{remark}

{
\begin{proposition}[Explicit confinement]\label{prop:lp-confinement}
 Fix $0<\kappa<p<1$ and a constant step $0<\beta<1/(\kappa L)$, where $L=2\|A\|_2^2$ and $1/0:=+\infty$. Let $\{x_k\}_{k}$ be generated by Algorithm~\ref{alg:pg-sqcv} from $x_1\in\mathbb R^n$. Set $E=F_\kappa(x_1)$ and choose
 \begin{equation}\label{eq:lp-radius}
 R>\max\left\{(E/\rho)^{1/p},(\|b\|_2^2/\rho)^{1/p}\right\}.
 \end{equation}
 All iterates and all global minimizers of either $F$ or $F_\kappa$ belong to $B(0,R)$.
\end{proposition}
\begin{proof}
 The proximal objective is continuous and coercive, so each proximal subproblem has a global minimizer. Proximal optimality with comparison point $x_k$ and the descent lemma give
 \[
 F_\kappa(x_{k+1})+
 \frac{1-\kappa\beta L}{2\beta}\|x_{k+1}-x_k\|^2
 \leq F_\kappa(x_k).
 \]
 Since $g,h\geq0$ and $\|x\|_2^p\leq\|x\|_p^p$, it follows that $\rho\|x_k\|_2^p\leq E$. Both $F$ and $F_\kappa$ are continuous and coercive and therefore attain their minima. If $x^*\in\argmin F$, then $\rho\|x^*\|_2^p\leq F(x^*)\leq F(0)=\|b\|_2^2$. If $z^*\in\argmin F_\kappa$, then $\rho\|z^*\|_2^p\leq F_\kappa(z^*)\leq F_\kappa(0)=\kappa\|b\|_2^2\leq\|b\|_2^2$. The strict choice of $R$ proves all assertions.
\end{proof}
}

\paragraph{The smooth part $g$.} Take $g(x)=\|Ax-b\|_2^2$. Then $g$ is convex and differentiable with $\nabla g(x) = 2A^\top(Ax-b)$ and $\nabla^2 g(x) = 2 A^\top A \succeq 0$, hence $g$ has $L$-Lipschitz gradient with
\begin{equation}\label{eq:L-value}
L = 2\|A\|_2^2.
\end{equation}
{Its strong-convexity modulus is $\mu=2\lambda_{\min}(A^\top A)\geq0$, which is positive exactly when $A$ has full column rank.}

\paragraph{The regularizer $h=\rho\|\cdot\|_p^p$.} By Proposition \ref{prop:lp-quasar-exact} {(a formulation on a Euclidean ball of \cite[Example~8, Proposition~6$(a)$]{BLL})}, $\overline{x}_{h}=0$, and for every $\kappa \in \, ]0, p[$ and every $R>0$, $\|x \|_p^p$ is $(\kappa, \gamma_0)$-strongly quasar-convex on $K = \overline{B}(0,R)$ with
\[
 \gamma_0 (\kappa, p, R) = \frac{2 (p-\kappa)}{ \kappa} R^{p-2}.
\]
By Lemma \ref{lem:scale} (scaling), $h= \rho \|\cdot\|_p^p$ is then $(\kappa,\gamma)$-strongly quasar-convex on $K$ with respect to the same $\overline{x}_{h}=0$, with modulus
\begin{equation}\label{eq:gamma-value}
 \gamma (\kappa,p,R,\rho) = \rho \gamma_0 (\kappa, p,R) = \frac{2\rho (p-\kappa)}{\kappa}R^{p-2}.
\end{equation}
As noted in \cite[Example~8]{BLL}, {for $n\geq2$ the penalty} $h$ is not quasiconvex. The quasar-convex framework of Section~{\ref{sec:03} applies to this penalty on the confinement ball; no quasar-convexity assumption} is {made on the full objective}.

For the remaining convergence results of this section, we use the following
application hypotheses, denoted collectively by (A):
\begin{enumerate}
\renewcommand{\theenumi}{A\arabic{enumi}}
\renewcommand{\labelenumi}{(\theenumi)}
\item\label{app:model} $A\in\mathbb R^{m\times n}$, $b\in\mathbb R^m$,
$0<p<1$, and $\rho>0$, with $g(x)=\|Ax-b\|_2^2$,
$h(x)=\rho\|x\|_p^p$, $F=g+h$, and $F_\kappa=\kappa g+h$.
Set $L=2\|A\|_2^2$ and $\mu=2\lambda_{\min}(A^\top A)\geq0$.
Here $\rho$ is the regularization weight and $\mu$ is the strong-convexity modulus of $g$.
\item\label{app:iteration} Fix $0<\kappa<p$ and a constant step $0<\beta_k \equiv \beta < 1/(\kappa L)$, with $1/0 := +\infty$. The sequence $\{x_k\}_{k}$ is generated by Algorithm~\ref{alg:pg-sqcv} from $x_1\in\mathbb R^n$, with arbitrary selections of global proximal minimizers, and is extended constantly if the algorithm stops.
Set $E=F_\kappa(x_1)$.
\item\label{app:ball} Choose $R$ satisfying \eqref{eq:lp-radius} and put $K=\overline B(0,R)$ and $\gamma=2\rho(p-\kappa)R^{p-2}/\kappa$, as in \eqref{eq:gamma-value}.
\end{enumerate}
By Proposition~\ref{prop:lp-quasar-exact} and Lemma~\ref{lem:scale}, $h$ is $(\kappa,\gamma)$-strongly quasar-convex on $K$ with respect to $0$
in the sense of Definition~\ref{def:quasar}.

Conditions such as positive restricted curvature or
$A^\top b=0$ are additional hypotheses only where explicitly stated.

The convergence comparisons retain the proximal residual, so no lower restriction on $\beta$ beyond positivity is imposed, and
the choice of $R$ imposes no further upper bound on $\beta$.

The bounds in Proposition~\ref{prop:lp-confinement} place every iterate,
accumulation point and global comparison minimizer strictly inside
$K=\overline B(0,R)$. The proximal comparisons use strong quasar-convexity
only along segments to zero and quadratic growth at comparison points
in $K$, while the radial subgradient inequality remains valid at interior
points. Descent gives $\sum_k\|x_{k+1}-x_k\|^2<+\infty$, and continuity
of $h$ supplies the value convergence needed for stationarity.
Consequently, the proofs of Proposition~\ref{prop:critical-comparison}
and Theorem~\ref{thm:exact} apply locally to the original unconstrained
iteration. When $A=0$, the same arguments apply with $L=\mu=0$;
the gradient map is then the identity, since $x-\kappa\beta\nabla g(x)=x$.

\subsection{{Geometric and Asymptotic Distance Estimates}}\label{sec:lp-distances}

We compare the iterates with global minimizers of $F=g+h$, both through a geometric envelope and through asymptotic distance bounds.
We bound $\|\nabla g(x_k)\|$ explicitly using the confinement $x_k\in\overline{B}(0,R)$:
\begin{align}
 \|\nabla g(x_k)\| & = 2\|A^\top (Ax_k-b)\| \le 2\|A\|_2 \big( \|A\|_2\|x_k\| + \|b\|_2 \big) \notag \\ 
 & \le 2 \|A\|_2 \big( \|A\|_2 R + \|b\|_2 \big) =: M. \label{eq:M-bound}
\end{align}

\begin{corollary}[Explicit distance bounds for power-regularized least squares]\label{cor:lp-exact}
Assume (\ref{app:model})--(\ref{app:ball}), and fix
$x^*\in\argmin_{\mathbb R^n}F$. Let $R_f$ and $C_\mu$ be the
constants in Theorem~\ref{thm:exact}, with $L$, $\mu$ and $\gamma$
as specified in (A). Set
\begin{equation}\label{eq:lp-geometric-rate}
 q=\frac{(2-\kappa)(1-\kappa\beta\mu)}
 {2-\kappa+\kappa\beta\gamma}\in (0,1).
\end{equation}
Then:
\begin{enumerate}
\item[$(a)$] For every $k\geq1$, we have
\begin{equation}\label{eq:lp-geometric-distance}
\begin{aligned}
 &\operatorname{dist}\bigl(x_k,\overline B(x^*,R_f)\bigr)\\
 &\quad\leq q^{k-1}
 \left(\|x_1-x^*\|_2+
 \frac{4(2-\kappa)}{\kappa\gamma}
 \|A^\top(Ax^*-b)\|_2-R_f\right)_+.
\end{aligned}
\end{equation}
\item[$(b)$] The asymptotic distance satisfies
\begin{equation}\label{eq:lp-asymptotic-distance}
 \limsup_k\|x_k-x^*\|_2
 \leq\min\left\{R_f,\ 2C_\mu
               \limsup_k\|A^\top(Ax_k-b)\|_2\right\}.
\end{equation}
\end{enumerate}
\end{corollary}

\begin{proof}
\noindent$(a)$ Proposition~\ref{prop:lp-confinement} places every iterate,
accumulation point and global comparison minimizer strictly inside
$K$, so the proof of Theorem~\ref{thm:exact} applies locally.
Since $\nabla g(x)=2A^\top(Ax-b)$, we have
$G^*=2\|A^\top(Ax^*-b)\|_2$. Substituting this expression and
the values of $L$, $\mu$ and $\gamma$ from (A) into
Theorem~\ref{thm:exact}$(a)$ gives $q$ and
\eqref{eq:lp-geometric-distance}.

\noindent$(b)$ With
$M_\infty=2\limsup_k\|A^\top(Ax_k-b)\|_2$,
Theorem~\ref{thm:exact}$(b)$ yields
\eqref{eq:lp-asymptotic-distance}.
\end{proof}

\begin{remark}
 The bound \eqref{eq:M-bound} supplies an explicit estimate of the gradient term in \eqref{eq:lp-asymptotic-distance} from the confinement radius. The estimates incorporate the strong-convexity modulus $\mu$, including when $A$ has full column rank. When $\mu=0$, the coefficient $2C_\mu$ reduces to
\[
 \frac{(2-\kappa)+\sqrt{2(2-\kappa)}}{\rho(p-\kappa)}R^{2-p}.
\]
The same localization of Proposition \ref{prop:critical-comparison} also gives
\[
 \|u-x^*\|_2\leq 2C_\mu\|A^\top(Au-b)\|_2
\]
for every critical point $u$ of $F_\kappa$ in $B(0,R)$ and every
global minimizer $x^*$ of $F$.
\end{remark}

{
\subsection{Support Identification and Convergence}\label{sec:lp-convergence}

We next use the threshold structure of the power penalty to establish convergence of the whole sequence to a critical point of $F_\kappa$. We write $\operatorname{supp}(x)=\{i:x_i\ne0\}$. For $S \subset \{1,\ldots,n\}$, $x_S$ is the vector of coordinates indexed by $S$ and $A_S$ is the matrix formed by the corresponding columns of $A$.

\begin{theorem}[Support identification and full-sequence convergence]\label{thm:lp-identification}
 Under (\ref{app:model})--(\ref{app:ball}), put $a_\beta=[2\beta\rho(1-p)]^{1/(2-p)}$. Every nonzero coordinate $x_{k,i}$ of $x_k$, with $k\geq2$, has magnitude at least $a_\beta$, that is, $|x_{k,i}| \geq a_\beta$. The support and coordinate signs therefore become constant after finitely many iterations. Denote the eventual support by $S$. Moreover,
 \[
 \sum_k\|x_{k+1}-x_k\| < + \infty,\qquad x_k \to \widehat x, \qquad 0 \in \kappa\nabla g(\widehat x) + \partial^Lh(\widehat x).
 \]
 These conclusions hold for every real $p \in \, ]0, 1[$.
\end{theorem}
\begin{proof}
 The proximal problem separates over coordinates. A nonzero scalar minimizer $u$ of $\rho|v|^p+(v-z)^2/(2\beta)$ satisfies $z=u+\beta\rho p\operatorname{sign}(u)|u|^{p-1}$. Comparing its value with the value at zero gives
 \[
 0\geq\rho|u|^p+\frac{u^2-2uz}{2\beta}
       =\rho(1-p)|u|^p-\frac{u^2}{2\beta},
 \]
 hence $|u|\geq a_\beta$. Since $\|x_{k+1}-x_k\|\to0$, transitions between zero and nonzero coordinates, which produce jumps of magnitude at least $a_\beta$, eventually cease; sign changes between nonzero coordinates produce jumps of magnitude at least $2a_\beta$ and also cease.
 If the final support $S$ is empty, the sequence is eventually zero. Otherwise, the restricted iterates $x_{k,S}$ eventually lie in a compact subset of a fixed sign orthant, with every active coordinate bounded away from zero. On this orthant
 \[
 \Phi_S(u)=\kappa\|A_Su-b\|_2^2+\rho\sum_{i\in S}|u_i|^p
 \]
 is real analytic for every real $p \in \, ]0, 1[$. Proximal optimality gives
 \[
 \|\nabla\Phi_S(x_{k+1,S})\|
 \leq(\beta^{-1}+\kappa L)\|x_{k+1}-x_k\|.
 \]
 Along the identified tail, $\Phi_S(x_{k,S}) = F_\kappa (x_k)$; continuity therefore gives the same limiting objective value at every accumulation point. Sufficient decrease, the relative-gradient estimate and the local analytic gradient inequality then give finite length and convergence by the KL convergence argument \cite[Theorem~2.9]{ABS}. The restricted gradient vanishes at the limit, while $\partial^L(\rho|\cdot|^p)(0)=\mathbb R$ on the complementary coordinates, proving the stated inclusion. The support and restricted-KL mechanisms are consistent with \cite{ZLWX,ZLX}.
\end{proof}

\begin{corollary}[Eventual linear rate]\label{cor:lp-local-linear}
 Assume (\ref{app:model})--(\ref{app:ball}). Let $S$ and $\widehat x$ be as in Theorem~\ref{thm:lp-identification}. If $S\ne\emptyset$ and
 \[
 H_S=2\kappa A_S^\top A_S-\rho p(1-p)
       \operatorname{Diag}(|\widehat x_i|^{p-2})_{i\in S}\succ0,
 \]
 then there exist $C>0$, $q_\ell \in \, ]0, 1[$ and $K$ such that $\|x_k-\widehat x\|\leq Cq_\ell^{k-K}$ for all $k \geq K$.
\end{corollary}

\begin{proof}
 Since $H_S=\nabla^2\Phi_S(\widehat x_S)$ is positive definite, $\Phi_S$ satisfies the KL inequality of Definition~\ref{def:KL} with exponent $\theta=1/2$ at $\widehat x_S$. Sufficient decrease and the preceding relative-gradient estimate then yield the linear-rate case of the same KL argument.
\end{proof}

}

{The common-minimizer case can be identified explicitly and gives finite termination.}

\begin{lemma}\label{lem:degenerate}
 {Under (\ref{app:model})}, if $A^\top b = 0$, then $x^* = \overline{x}_{h} = 0$ is the unique global minimizer of $F=g+h$ over $\mathbb{R}^n$.
 {Conversely, if $x^*\in\argmin F$ and $\nabla g(x^*)=0$, then $x^*=0$ and $A^\top b=0$.}
\end{lemma}

\begin{proof}
 Since $A^\top b=0$, $\nabla g(0) = -2A^\top b = 0$; as $g$ is convex, $0 \in \argmin_{\R^n} g$, i.e., $g(0)=g^*=\|b\|_2^2$. For every $x\ne0$, $g(x)\ge g^*=g(0)$ and $h(x)=\rho\|x\|_p^p > 0 = h(0)$ (as $p>0$), so $F(x)=g(x)+h(x) > g(0)+h(0)=F(0)$. Hence $0$ is the unique global minimizer of $F$, and since $\overline{x}_{h}=0$ as well, $x^*=\overline{x}_{h}=0$.
 {For the converse, if $x^*\in\argmin F$ and $\nabla g(x^*)=0$, then a nonzero coordinate would satisfy the necessary condition $\rho p\operatorname{sign}(x_i^*)|x_i^*|^{p-1}=0$, a contradiction. Thus $x^*=0$ and $A^\top b=0$.}
\end{proof}

\begin{corollary}[{Finite termination in the case $A^\top b=0$}] \label{cor:lp-exact-recovery}
 {Assume (\ref{app:model})--(\ref{app:ball}).} Suppose $A^\top b=0$.
 {Then the algorithm reaches zero in finitely many iterations. More precisely, any $k\geq2$ such that
 \[
  q^{k-1}\|x_1\|_2<a_\beta,
 \]
 with $q$ from \eqref{eq:lp-geometric-rate} and $a_\beta$ from Theorem~\ref{thm:lp-identification}, satisfies $x_j=0$ for every $j\geq k$.}
\end{corollary}

\begin{proof}
 By Lemma \ref{lem:degenerate}, $x^*=0$ and $G^*=0$. Thus $R_f=0$, and Corollary~\ref{cor:lp-exact}$(a)$ gives $\|x_k\|_2 \leq q^{k-1} \|x_1\|_2$. The nonzero-coordinate gap in Theorem \ref{thm:lp-identification} therefore forces $x_k=0$ whenever the stated inequality holds. Since $0<q<1$ and $a_\beta > 0$, such an iteration is reached in finite time. The next proximal subproblem has input zero and unique minimizer zero, so all subsequent iterates remain zero.
\end{proof}

\section{Conclusions}\label{sec:conclusions}

We developed quantitative proximal-gradient comparisons for a convex smooth term and a possibly nonconvex strongly quasar-convex regularizer. The arbitrary-reference proximal inequality yields optimized composite distance estimates. A radial argument also gives a geometric envelope around every global minimizer of the unweighted objective, with a radius independent of the step size. The critical-point comparison extends the best asymptotic comparison constant to all values of $\mu\geq0$, and a common minimizer yields global linear convergence. Critical-point convergence concerns the weighted objective decreased by the iteration, while the distance estimates use global minimizers of the unweighted objective as reference points. The weighting example shows that the two notions of criticality can differ.

For power-regularized least squares, descent provides the ball needed for local strong quasar-convexity. Identification of the support and signs after finitely many iterations, combined with analyticity on the identified support, yields whole-sequence convergence to a critical point of the weighted objective for every real power between zero and one. Positive restricted curvature yields an eventual linear rate. Future work may address inexact proximal solutions and adaptive parameter selection while preserving these comparisons.

\section{Declarations}

%\subsection{Ethical Approval and Consent to participate}

%Not Applicable.

%\subsection{Consent for publication}

%Not Applicable.

%\subsection{Human and Animal Ethics}

%Not Applicable.

\subsection{Availability of Supporting Data}

No data sets were generated during the current study.

\subsection{Author Contributions}

 All authors contributed equally to the conception, design and implementation of the study and drafted and revised the manuscript.

\subsection{Conflicts of Interest}

There are no conflicts of interest or competing interests related to this manuscript.

\subsection{Funding}

This work has been partially su\-ppor\-ted by ANID---Chile through Fondecyt Regular 1241040 (Lara).

% Manual bibliography.

\end{document}